%% file: Jasour-Arxive.tex
\documentclass[letterpaper, 10 pt, conference]{ieeeconf}{\onecolumn}  

\usepackage{fullpage,enumerate,setspace,changepage,multirow,rotating,xcolor}
\usepackage{graphicx,amsmath,amssymb,amsfonts,url, algorithm, algorithmic, etoolbox}
\usepackage{array,subfigure}
\usepackage[small]{caption}
\input defs.tex

\newtheorem{lemma}{Lemma}
\newtheorem{theorem}{\textbf{Theorem}}
\usepackage{amsmath}
\usepackage{etoolbox}
\AtBeginEnvironment{align}{\setcounter{subeqn}{0}}
\newcounter{subeqn} \renewcommand{\thesubeqn}{\theequation\alph{subeqn}}%
\newcommand{\subeqn}{%
	\refstepcounter{subeqn}
	\tag{\thesubeqn}
}

\def\norm#1{\|#1\|}

\def\norm#1{\|#1\|}

\newcommand{\seq}{\reals^{\mathbb{N}}}
\makeatletter
\newcommand{\thickhline}{%
	\noalign {\ifnum 0=`}\fi \hrule height 1pt
	\futurelet \reserved@a \@xhline
}

\IEEEoverridecommandlockouts                              
\title{\LARGE \bf
Convex Chance Constrained Model Predictive Control
}

\author{Ashkan Jasour and Constantino Lagoa
\thanks{This work was partially supported by the National
	Science Foundation under grants CMMI-1400217, CNS-1329422, and ECCS-1201973}
\thanks{Ashkan Jasour is with the Department of Electrical Engineering, The Pennsylvania State University,
        {\tt\small jasour@psu.edu}}%
\thanks{Constantino Lagoa is with the Department of Electrical Engineering, The Pennsylvania State University,
       Sc, PA, USA
        {\tt\small lagoa@psu.edu}}%
}

\begin{document}

\maketitle
\thispagestyle{empty}
\pagestyle{empty}

\begin{abstract}

We consider the \textit{Chance Constrained Model Predictive Control} problem for polynomial systems subject to disturbances. In this problem, we aim at finding optimal control input for given disturbed dynamical system to minimize a given cost function subject to probabilistic constraints, over a finite horizon. The control laws provided have a predefined (low) risk of not reaching the desired target set. Building on the theory of measures and moments, a sequence of finite semidefinite programmings are provided, whose solution is shown to converge to the optimal solution of the original problem. Numerical examples are presented to illustrate the computational performance of the proposed approach.

\end{abstract}

\section{INTRODUCTION}

In this paper, we aim at solving chance constrained model predictive control (CCMPC) problems whose objective is to obtain finite-horizon optimal control of dynamical systems subject to probabilistic constraints. The control laws provided are designed to have precise bounds on the probability of achieving the desired objectives. More precisely, consider a polynomial dynamical system subject to external perturbation and assume that the probability distribution of the disturbances at each time is known. Then, given a desired set defined by polynomial inequalities and a polynomial cost function defined in terms of states and control input of the system, we aim at designing a controller to i) minimize the expected value of given cost function over the finite horizon and ii) reach the given desired set with high probability. For this purpose, at each sampling time we solve a convex optimization problem that minimizes the expected value of cost function subject to probabilistic constraints over the finite horizon.

Probabilistic formulations of model predictive control  such  as  the  one  above can be  used  in  different  areas  to  deal  with systems subject to disturbances. A few examples are probabilistic obstacle avoidance in motion planning of robotic systems  under  environment uncertainty \cite{Ref_Pro2}, risk  management problem \cite{Ref_Risk} and macroeconomic system control in  the  area of economy,  finance \cite{Ref_Econ1}, energy management problems \cite{Ref_Energy} and many  other areas that can be formulated as instances of CCMPC problems.  Although in some particular cases chance constraints problems are convex \cite{Ref_Conve}, in general, these problems are not convex. In this paper, using the theory of measures and moments we provide a semidefinite program whose  solution  converges  to  the  solution of the CCMPC problem mentioned above.

\subsection{Previous Work}

The MPC method is an optimal control based method, which a finite cost function is optimized at every sampling time under imposed constraints. At each sampling time, MPC needs to predict the future states of the system over the finite horizon using the dynamic of the system. To deal with uncertain parameters of the system and disturbance, several approaches have been proposed.

In (\cite{Ref_R1,Ref_R2}), robust MPC for linear and polynomial systems are proposed where robust constraints are employed. In this method, MPC is formulated considering the a bundle of trajectories for all possible realizations of the uncertainty.
The robust MPC methods are conservative, due to the requirement of robust feasibility for all disturbance realizations.

In (\cite{Ref_Ad1,Ref_Ad2,Ref_Ad3}), adaptive MPC are provided where neural networks are used to predict the future behavior of the system.  Using the online training algorithm, robustness against changes in the robot parameters is obtained.

In (\cite{Ref_Pro1,Ref_Pro2,Ref_Pro3}), to deal with model uncertainty the probabilistic constraints are used. In (\cite{Ref_Pro1,Ref_Pro2}) probabilistic constraints for linear systems are replaced with hard constrained assuming the Gaussian distribution for uncertainty. In \cite{Ref_Pro3}, a semialgabriac approximation of the probabilistic constraints are obtained. 

In this paper, take a different approach to deal with chance constraints. The proposed method is based on chance constrained optimization method that we have presented in (\cite{Ref_Chance, Ref_Chance2}). In this method, the relaxed optimization is provided in measure and moment space. One needs to search for the positive Borel measure on the given semialgebraic set, while simultaneously searching for an upper bound probability measure over a simple set containing the semialgebraic set and restricting the Borel measure.

\subsection{The Sequel}
The outline of the paper is as follows: in Section II, the notation adopted in the paper and preliminary results
on measure and moment theory are presented. In Section III, we precisely define the chance constrained MPC problem. In
Sections IV, we provide equivalent infinite dimensional convex problem one measure and in Section V we provide a semidefinite program on moments to solve obtained convex problem on measures. In Section VI,
some numerical results are presented to illustrate the numerical performance of the proposed approach, and finally,
conclusion is stated in Section VII.

\section{Notation and Preliminary Results}

Let $\mathbb{R}[x]$ be the ring of real polynomials in the variables $x \in \mathbb{R}^n$. Given $\cP\in\mathbb{R}[x]$, we will represent $\cP$ as $\sum_{\alpha\in\mathbb{N}^n} p_\alpha x^\alpha$ using the standard basis $\{x^\alpha\}_{\alpha\in \mathbb{N}^n}$ of $\mathbb{R}[x]$, and $\mathbf{p}=\{p_\alpha\}_{\alpha\in\mathbb{N}^n}$ denotes the polynomial coefficients. We assume that the elements of the coefficient vector $\mathbf{p}=\{p_\alpha\}_{\alpha\in\mathbb{N}^n}$ are sorted according to grevlex order on the corresponding monomial exponent $\alpha$.

Given $n$ and $d$ in $\mathbb{N}$, we define $S_{n,d} := \binom{d+n}{n}$ and $\mathbb{N} ^{\rm n}_d := \{\alpha \in \mathbb N^n : \norm{\alpha}_1 \leq d \}$. Let $\mathbb R_{\rm d}[x] \subset \mathbb R [x]$ denote the set of polynomials of degree at most $d\in \mathbb{N}$, which is indeed a vector space of dimension $S_{n,d}$. Similarly to $\cP\in\mathbb{R}[x]$, given $\cP\in\mathbb R_{\rm d}[x]$, $\mathbf{p}=\{p_\alpha\}_{\alpha\in\mathbb{N}^{\rm n}_d}$ is sorted such that $\mathbb{N} ^{\rm n}_d\ni\mathbf{0} = \alpha ^{(1)} <_g \ldots <_g \alpha ^{(S_{n,d})}$, where $S_{n,d}$ is the number of components in $\mathbf{p}$. 


Let $\seq$ denote the vector space of real sequences. Given $\mathbf{y}=\{y_\alpha\}_{\alpha\in\mathbb{N}^n}\subset\seq$, 
let $L_\mathbf{y}:\reals[x]\rightarrow\reals$ be a linear map defined as (\cite{Ref_Lass1, Ref_Lass2})
\begin{small}
	\begin{equation}
	\label{eq:lin_map}
	\cP \quad \mapsto \quad L_\mathbf{y}(\cP)=\sum_{\alpha\in\mathbb{N}^n}p_\alpha y_\alpha, \quad \hbox{where} \quad \cP(x)=\sum_{\alpha\in\mathbb{N}^n} p_\alpha x^\alpha
	\end{equation}
\end{small}
A sequence $\mathbf y = \{ y_ \alpha \}_{\alpha\in\mathbb{N}^n}\in\seq$
is said to have a \emph{representing measure}, if there exists a finite Borel measure $\mu$ on $\reals^n$ such that $y_{\alpha } = \int{x^{\alpha} d\mu }$ for every $\alpha \in \mathbb N ^n$, (\cite{Ref_Lass1, Ref_Lass2}). In this case, $ \mathbf y $ is called the moment sequence of the measure $\mu $.  

\textbf{Moment Matrix:} Given $r\geq 1$ and the sequence $\{y_\alpha\}_{\alpha\in\mathbb{N}^n}$, the moment matrix $M_r({\mathbf y})\in\reals^{S_{n,r}\times S_{n,r}}$, containing all the moments up to order $2r$, is a symmetric matrix and its $(i,j)$-th entry is defined as follows (\cite{Ref_Lass1, Ref_Lass2}):
\begin{equation}\label{momnt matirx def}
M_r ( \mathbf y )(i,j):= L_{\mathbf y}\left(x^{\alpha^{(i)}+\alpha^{(j)}}\right)=y_{\alpha^{(i)}+\alpha^{(j)}}\ \ \
\end{equation}
where  $1 \leq i,j \leq S_{n,r}$, $\mathbb{N} ^{\rm n}_r\ni\mathbf{0} = \alpha ^{(1)} <_g \ldots <_g \alpha ^{(S_{n,2r})}$ and $S_{n,2r}$ is the number of moments in $\mathbb{R}^n$ up to order $2r$. Let $\cB_r^T=\left[x^{\alpha^{(1)}},\ldots, x^{\alpha^{(S_{n,r})}}\right]^T$ denote the vector comprised of the monomial basis of $\mathbb R_{\rm r}[x]$. Note that the moment matrix can be written as $M_r({\mathbf y}) = L_\mathbf{y}\left(\cB_r \cB_r^T\right)$; here, the linear map $L_\mathbf{y}$ operates componentwise on the matrix of polynomials, $\cB_r \cB_r^T$. For instance, let $r=2$ and $n=2$; the moment matrix containing moments up to order $2r$ is given as
\begin{equation} \label{moment matrix exa}
M_2\left({\mathbf y}\right)=\left[ \begin{array}{c}

\begin{array}{ccc} y_{00} \ | & y_{10} & y_{01}| \end{array}
\begin{array}{ccc} y_{20} & y_{11} & y_{02} \end{array}
\\
\begin{array}{ccc} - & - & - \end{array}
\ \ \ \  \begin{array}{ccc} - & - & - \end{array}
\\

\begin{array}{ccc} y_{10}\ | & y_{20} & y_{11}| \end{array}
\ \begin{array}{ccc} y_{30} & y_{21} & y_{12} \end{array}
\\

\begin{array}{ccc} y_{01}\ | & y_{11} & y_{02}| \end{array}
\ \begin{array}{ccc} y_{21} & y_{12} & y_{03} \end{array}
\\

\begin{array}{ccc} - & - & - \end{array}
\ \ \ \ \  \begin{array}{ccc} - & - & - \end{array}
\\

\begin{array}{ccc} y_{20}\ | & y_{30} & y_{21}| \end{array}
\ \begin{array}{ccc} y_{40} & y_{31} & y_{22} \end{array}
\\

\begin{array}{ccc} y_{11}\ | & y_{21} & y_{12}| \end{array}
\ \begin{array}{ccc} y_{31} & y_{22} & y_{13} \end{array}
\\

\begin{array}{ccc}y_{02}\ | & y_{12} & y_{03}| \end{array}
\ \begin{array}{ccc} y_{22} & y_{13} & y_{04} \end{array}

\end{array}
\right]
\end{equation}
\textbf{Localizing Matrix:} Given a polynomial $\mathcal{P} \in \mathbb R [x]$, let $ \mathbf p = \{ p_{\gamma }\}_{\gamma\in\mathbb{N}^n}$ be its coefficient sequence in standard monomial basis, i.e., $\cP(x)=\sum_{\alpha\in\mathbb{N}^n} p_\alpha x^\alpha$, 
the $(i,j)$-th entry of the \emph{localizing matrix} $M_r(\mathbf{y};\mathcal{P})\in\reals^{S_{n,r}\times S_{n,r}}$ with respect to $\mathbf y $ and $\mathbf p$ is defined as follows (\cite{Ref_Lass1, Ref_Lass2}):
\begin{small}
	\begin{equation}\label{localization matrix def}
	M_r(\mathbf y;\mathcal{P})(i,j) := L_{\mathbf{y}}\left(\cP x^{\alpha^{(i)}+\alpha^{(j)}}\right)=\sum_{\gamma \in \mathbb N^n} p_{\gamma} y_{\gamma +\alpha^{(i)}+\alpha^{(j)}} \ \ 
	\end{equation}
\end{small}where, $1 \leq i,j \leq  S_{n,d}$. Equivalently, $M_r(\mathbf y,\mathcal{P}) = L_{\mathbf{y}}\left(\mathbf \cP \cB_r\cB_r^T\right)$, where $L_{\bf y}$ operates componentwise on $\cP \cB_r\cB_r^T$. For example, given $\mathbf{y}=\{y_\alpha\}_{\alpha\in\mathbb{N}^2}$ and the coefficient sequence $\mathbf{p}=\{p_\alpha\}_{\alpha\in\mathbb{N}^2}$ corresponding to polynomial $\cP$,
\begin{equation}
\mathcal{P}(x_1,x_2)= bx_1-cx^2_2,
\end{equation}
the localizing matrix for $r=1$ is formed as follows
\begin{equation}
M_1(\mathbf{y};\mathcal P)= \begin{small}
\left[ \begin{array}{ccc}
by_{10}-cy_{02} & by_{20}-cy_{12} & by_{11}-cy_{03} \\
by_{20}-cy_{12} & by_{30}-cy_{22} & by_{21}-cy_{13} \\
by_{11}-cy_{03} & by_{21}-cy_{13} & by_{12}-cy_{04} \end{array}
\right]
\end{small}
\end{equation}

Let $C\subset\reals^n$, $\Sigma(C)$ denotes the Borel $\sigma$-algebra over $C$. Given two measures $\mu_1$ and $\mu_2$ on a Borel $\sigma $-algebra $\Sigma $, the notation $\mu_1 \preccurlyeq \mu _2$ means  $\mu _1(S)\le \mu_2(S)$ for any set $S \in \Sigma $. Moreover, if $ \mu_1$ and $ \mu_2$ are both measures on Borel $\sigma$-algebras $\Sigma_1$ and $\Sigma_2$, respectively, then $\mu =\mu_1 \times \mu_2$ denotes the product measure satisfying $ \mu(S_1 \times S_2)=\mu_1 (S_1) \mu_2(S_2)$ for any measurable sets $S_1\in \Sigma_1$, $S_2 \in \Sigma_2$ \cite{Ref_Vol}. Also, let $\cM_+(\chi)$ be the space of finite nonnegative Borel measures $\mu$ such that $supp(\mu)\subset\chi$, where $supp(\mu)$ denotes the support of the measure $\mu$; i.e., the smallest closed set that contains all measurable sets with strictly positive $\mu$ measure, \cite{r_Supp}. Given two square symmetric matrices $A$ and $B$, the notation $A \succcurlyeq 0$ denotes that $A$ is positive semidefinite, and $A\succcurlyeq B$ stands for $A-B$ being positive semidefinite.

\textbf{Moment Condition:} The following lemmas give necessary and sufficient conditions for a moment sequence $\mathbf y$ to have a representing measure $\mu$; for details see (\cite{Ref_Lass1, Ref_Lass2}).
\begin{lemma}
	\label{lem1}
	Let $\mu$ be a finite nonnegative Borel measure on $\reals^n$ and $\mathbf{y}=\{y_\alpha\}_{\alpha\in\mathbb{N}^n}$ such that $y_\alpha=\int x^\alpha d\mu$ for all $\alpha\in\mathbb{N}^n$. Then $M_d(\mathbf y)\succcurlyeq 0$ for all $ d\in \mathbb N$.
\end{lemma}

Given polynomials $\mathcal{P}_j\in \mathbb R [x], j=1,\dots,\ell$, consider the semialgebraic set $\cK$ defined as
\begin{equation}\label{preliminary result_semi algebraic set}
\cK = \{ x\in \mathbb{R}^n: \mathcal{P}_j(x)\geq0,\ j=1,2,\dots ,\ell\ \}.
\end{equation}

\begin{lemma}
	\label{lem2}
	If $\cK$ defined in \eqref{preliminary result_semi algebraic set} satisfies Putinar's property, 
	then the sequence $\mathbf y = \{ y_\alpha\}_{\alpha\in\mathbb{N}^n}$ has a \emph{representing} finite nonnegative Borel measure $\mu$ on the set $\cK$, if and only if
	\begin{equation*}
	M_d(\mathbf y)\succcurlyeq 0,\quad M_d(\mathbf y;\mathcal{P}_j)\succcurlyeq 0,\ \ j=1,\dots ,\ell, \forall \  d\in \mathbb N.
	\end{equation*}
	If $\cK \subset [-1, 1]^n$, the condition $M_d(\mathbf y)\succcurlyeq 0,  \forall \  d\in \mathbb N$ is sufficient.
\end{lemma}

Finally, the following lemma, proven in~\cite{Ref_Vol}, shows that the Borel measure of a compact set is equal to the optimal value of an infinite dimensional LP problem.
\begin{lemma}
	\label{preliminary result_volume}
	Let $\Sigma$ be the Borel $\sigma$-algebra on $\reals^n$, and $ \mu_1$ be a measure on a compact set $\cB\subset\Sigma$. Then for any given $\cK\in\Sigma$ such that $\cK\subseteq  \cB$, one has
	\begin{equation*}
	\mu_1(\cK)= \int_{\cK} d\mu_1 = \sup_{\mu_2\in\cM(\cK)} \left\lbrace  \int d\mu_2 : \mu_2 \preccurlyeq \mu_1\right\rbrace,
	\end{equation*}
	where $\cM(\cK)$ is the set of finite Borel measures on $\cK$.
\end{lemma}

\section{Problem Formulation}
In this paper, we consider \emph{chance constrained model predictive control} problem defined as follows. Consider the following discrete-time stochastic dynamical system 
\begin{equation}\label{sys1}
x_{k+1}= f(x_k,u_k,\omega_k)
\end{equation}
where $f: \reals^{n_x+n_u+n_{\omega}} \rightarrow \reals^{n_x}$ is a polynomial function,  $x_k \in \chi \subseteq \reals^{n_x}$ is system state, $u_k \in \psi \subseteq \reals^{n_u}$ is control input, and $\omega_k\in \Omega \subseteq R^{m_\omega}$ is disturbance, at time step $k$. The disturbances $\omega_k$ at time $k$ are independent random variables with probability measure $\mu_{\omega_k}$ supported on $\Omega $, respectively. We assume that $\Omega $ is compact semialgebraic set of the form
$\Omega = \lbrace \omega \in \reals^{n_{\omega}} : \mathcal{P}_{\omega}(\omega) \geq 0 \rbrace$
for given polynomial $\mathcal{P}_{\omega}$. 
Also, let $\chi_N$ be a given desired set defined by the compact semialgebraic sets as 
 \begin{equation}\label{Pro_set}
 \chi_D = \lbrace x \in \chi : \mathcal{P}_{\chi_D}(x) \leq 0 \rbrace
 \end{equation}
 
In this paper we aim at solving following problem.

\textbf{Problem 1}: For a given stochastic dynamical system in \eqref{sys1}, find an optimal control $u$ to:\\
i) Reach the desired set $\chi_D$ with high probability,\\
ii) Minimize the expected value of given cost function in terms of states and inputs of the system.

To obtain such control input, at each sampling time $k$, we solve the following optimization problem:

{\small \begin{align} \label{Problem2}
& \mathbf {P_{MPC}^*} :=\ \min_{u \in \cU} \hbox{E}\left[ \cP_{cost}\left( \{x_i\}_{i=k+1}^{k+N_p},\{u_i\}_{i=k}^{k+N_p}\right) \right] \\
& \hbox{s.t.}\quad \nonumber \\
& \hbox{Prob}_{\mu_{\omega_k}}\left\lbrace \cP_{\chi_D} (x_{k+1}) \geq \alpha \cP_{\chi_D} (x_{k}) \right\rbrace \geq 1- \beta \cP_{\chi_D} (x_{k}) \label{Problem2_con} \subeqn \\
& x_{k+1}= f(x_k,u_k,\omega_k), \ \ \{\omega_i \sim \mu_{\omega_i} \}_{i=k}^{k+N_p-1}\subeqn
\end{align}}where, $u = \{u_i\}_{i=k}^{k+N_p} \in \cU \subset \reals^{N_p}$ is sequence of inputs, $E[.]=\int (.) d\mu_{\omega_k}...d\mu_{\omega_{k+N_p-1}}$ is expected value operator, $N_p \geq 1 \in \mathbb{N} $ is prediction horizon. $0< \alpha<1 $ and $0< \beta<1 $ such that $ 0 \leq \beta \cP_{\chi_D}(x) <1$ for all $x\in \chi$. Polynomial $\cP_{cost}\left( \{x_i\}_{i=k+1}^{k+N_p},\{u_i\}_{i=k}^{k+N_p}\right) $ is defined cost function in terms of states and control input of the system over control and prediction horizon. We assume that the set of feasible control input $\cU$ is a semialgebraic set defined as
\begin{equation}\label{U}
	\cU :=\left\{u=(u_{k},...,u_{k+N_p}):\cP_{\cU} (u) \geq 0 \right\}
\end{equation}
Also, using the dynamic of the system in \eqref{sys1}, $\{x_i\}_{i=k+1}^{k+N_p}$, sequences of system states over the prediction horizon, can be explicitly expressed in terms of disturbance and input of the system as
 \begin{equation}
 	x_i = \cP_{x_i}(\{u_j\}_{j=k}^{i-1},\{\omega_j\}_{j=k}^{i-1}) \ \ i=k+1,...,N_p
 \end{equation}
Then, expected value in the cost function \eqref{Problem2} can be rewritten in terms of inputs as
 \begin{equation}
 	\hbox{E}\left[ \cP_{cost}\left( \{x_i\}_{i=k+1}^{k+N_p},\{u_i\}_{i=k}^{k+N_p}\right) \right] = \cP_{E}(u) 
 \end{equation}
where, $\cP_{E}: \reals^{N_p} \rightarrow \reals $ is a polynomial function and $u=\{u_i\}_{i=k}^{k+N_p}$.

By solving problem in \eqref{Problem2}, we find sequence of control inputs $\{u_i\}_{i=k}^{k+N_p}$ that minimizes expected value of defined cost function over the finite horizon with respect to the chance constraint \eqref{Problem2_con}. Chance constraint \eqref{Problem2_con} implies that the probability of getting closer to the desired set at next sampling time $k+1$ is bounded with respect to $\cP_{\chi_D} (x_{k})$, the distance of states of the system to the desired set at current time $k$.
At each sampling time $k$, the first element of the obtained control input $u$ is applied to the system \cite{Ref_Ad4}. The implemented chance constraint \eqref{Problem2_con} depend only on $u_k$; hence, is recursively feasible. \\
\textbf{Assumption:} We assume that for every $x \in \chi$, there exist a $u$ such that the probability constraint \eqref{Problem2_con} is satisfied. Hence, problem \eqref{Problem2} is always feasible.

The following theorem holds true.
\begin{theorem}\label{Theo1}
Given an initial state $x_0 \in \chi$ and $\epsilon > 0$ there exist a $\hat{k}(\epsilon, \alpha, \beta)$ and $\hat{P}(\epsilon, \alpha, \beta)$ such that
\begin{equation}
 \hbox{Prob} \left\lbrace  \cP_{\chi_D} (x_{k}) \leq \epsilon, \ \forall k \geq \hat{k}(\epsilon,\alpha,\beta) \right\rbrace \geq  \hat{P}(\epsilon,\alpha,\beta)
\end{equation}
where,
\begin{equation}
\hat{k}(\epsilon, \alpha, \beta) \geq \frac{\hbox{ln}(\epsilon)-\hbox{ln}(\cP_{\chi_D} (x_{0}) )}{\hbox{ln}(\alpha)}
\end{equation}
\begin{equation}
 \hat{P}(\epsilon, \alpha, \beta)=\prod_{i=0}^{\hat{k}-1}(1- \beta \alpha^i) > 0 \label{PB}
\end{equation}

\end{theorem}

\begin{proof}
See Appendix \ref{Appen_Theo1}.	
\end{proof}
The probability lower bound \eqref{PB} is a convergent product and converges to a non-zoro constant.
For example, consider the cases that $(\alpha,\beta)=(0.8,0.05)$. For this case, $\hat{P}$ converges to 0.8169 for $\hat{k} \geq 36$. In the section \ref{Sec:exa}, where numerical examples are presented, we consider this case for $\alpha$ and $\beta$.

\textbf{Remark:} The lower bound probability \eqref{PB} is conservative bound and the actual probability of reaching the $\epsilon$ level set of $\cP_{\chi_D}$ is greater than provided $\hat{P}(\epsilon,\alpha,\beta)$. However, lower bound \eqref{PB} is useful for controller design purposes and shows that the probability of reaching the set is nonzero.

The provided problem in \eqref{Problem2} is in general non convex and hard to solve. In the next section, we provide a convex equivalent problems to the problem \eqref{Problem2}.

\section{Equivalent Convex Problem on Measures}

As an intermediate step in the development of finite convex relaxations of the original problem in \eqref{Problem2}, a related infinite dimensional convex problem on measures is provided as follows. Let $\mu_u$ and $\mu$ be the finite nonnegative Borel measures and also the set $\cK$ be defined as
\begin{equation}\label{K1}
\mathcal K :=\left\{(u_{k},\omega_k):\cP_{\chi_D} (x_{k+1}) - \alpha \cP_{\chi_D} (x_{k}) \geq 0\right\}
\end{equation}
$\ \ \ \ \ \ \ \ \ \ \ \ \ \ \ \ \ =\left\{(u_{k},\omega_k):\cP_{\cK} (u_{k},\omega_k) \geq 0\right\}$\\
where, polynomial $\cP_{\cK}$ can be obtained using system dynamics and and polynomial $\cP_{\chi_D}$.
Consider the following convex problem on measures:
\begin{align}
\mathbf{P_{measure}^*}:=&\ \sup_{\mu ,\mu_u} \int  \cP_{E}(u)  d\mu_u, \label{Problem3}\\
& \hbox{s.t.}\  \int d\mu \geq 1- \beta\cP_{\chi_D} (x_{k}) \label{Problem3_1}\subeqn \\
& \mu \preccurlyeq \mu_u \times \Pi_{i=k}^{k+N_p-1}\mu_{\omega_i}, \label{Problem3_2}\subeqn\\
& \int \mu_u = 1, \label{Problem3_3}\subeqn\\
& \mu\in\cM_+(\mathcal K),\ \mu_u\in\cM_+(\cU).  \label{Problem3_4} \subeqn
\end{align}
where, measures $\mu$ and $\mu_u$ are supported on the sets $\cU$ and $\cK$ defined as \eqref{U} and \eqref{K1}.

Assume that there exist a unique solution $u^* \in \cU$ to the problem in \eqref{Problem2}. Then, following theorem shows the equivalency of the problem in  \eqref{Problem3} and the original volume problem in \eqref{Problem2}.

\begin{theorem} \label{Theo 2}
	Assume that $\mu_u^*$, the solution of the problem \eqref{Problem3}, is a delta distribution whose mass is concentrated on a single point $u^*$. Then, optimization problem in \eqref{Problem2} is equivalent to the infinite LP in \eqref{Problem3} in the following sense:
	\begin{enumerate}[i)]
		\item The optimal values are the same, i.e., $\mathbf{P_{MPC}^*}=\mathbf{P_{measure}^*}$.
		\item $u^* \in supp(\mu_u^*)$ is an optimal solution to \eqref{Problem2}.
		\item If an optimal solution to \eqref{Problem2} exists, call it $u^*$, then $\mu_u = \delta_{u^*}$, delta measure at $u^*$, and $\mu = \delta_{u^*} \times \Pi_{i=k}^{k+N_p-1}\mu_{\omega_i}$ is an optimal solution to \eqref{Problem3}.
	\end{enumerate}
\end{theorem}

\begin{proof}
	See Appendix \ref{Appen_Theo2}.
\end{proof}

In the next section, we provide the tractable finite relaxations to the problem \eqref{Problem3}.

\section{Semidefinite Programming relaxations on Moments}

In this section, we provide an finite dimensional semidefinite programming (SDP) of which feasible region is defined over real sequences. We show that the corresponding sequence of optimal solutions can arbitrarily approximate the optimal solution of \eqref{Problem3}, which characterizes the optimal solution of original problem in \eqref{Problem2}.
Unlike the problem \eqref{Problem3} in which we are looking for measures, in the SDP formulation given in \eqref{Problem4}, we aim at finding moment sequences corresponding to measures that are optimal to \eqref{Problem3}. 
Consider the following finite dimensional SDP:
\begin{align}
\mathbf{P^*_r}:= &\sup_{\mathbf{y}\in\reals^{S_{(N_p-1)n_{\omega}+N_p,2r}},\ \mathbf{y_u}\in\reals^{S_{N_p,2r}}} L_\mathbf{y_u}(\cP_{E}\left( u)\right),
\label{Problem4}\\
\hbox{s.t.}\quad & M_r(\mathbf y)\succcurlyeq 0,\ M_{r-r_{\cK}}(\mathbf{y}; \cP_{\cK})\succcurlyeq 0, \label{Problem4_1}\subeqn\\
& \left(\mathbf{y}\right)_\mathbf{0} \geq  1- \beta\cP_{\chi_D} (x_{k}), \ \left(\mathbf{y_u}\right)_\mathbf{0}=1, \label{Problem4_2}\subeqn\\
&M_r ({\mathbf y}_{\mathbf u})\succcurlyeq 0,\ M_{r-r_{\cU}}(\mathbf{y}_{\mathbf u}; \mathcal{P}_{\mathcal{U}})\succcurlyeq 0, \label{Problem4_3}\subeqn\\
&M_r (\mathbf{y_u}\times\Pi_{i=k}^{k+N_p-1}\mathbf{y_{\omega_i}}-{\mathbf y})\succcurlyeq 0.\label{Problem4_4}\subeqn
\end{align}
where $L_\mathbf{y_u}$ is the linear map defined in \eqref{eq:lin_map}. $\left(\mathbf{y}\right)_\mathbf{0}$ and $\left(\mathbf{y_u}\right)_\mathbf{0}$ are first element of the sequences $\mathbf{y}$ and $\mathbf{y_u}$, respectively. Polynomials  $\cP_{\cU}$ and $\cP_{\cK}$ are defined in \eqref{U} and \eqref{K1}. $r\in\integers_+$ is relaxation order of matrices, $d_{\cK}$ and $d_{\cU}$ are the degree of polynomial $\cP_{\cK}$ and $\cP_{\cU}$, $r_{\cK}:=\left\lceil\frac{d_{\cK}}{2}\right\rceil$ and $r_{\cU}:=\left\lceil\frac{d_{\cU}}{2}\right\rceil$. Also, $\mathbf{y_u}\times \Pi_{i=k}^{k+N_p-1}\mathbf{y_{\omega_i}}$ is truncated moment sequence of measure $\mu_u \times \Pi_{i=k}^{k+N_p-1}\mu_{\omega_i}$. $M_{r-r_{\cK}}(\mathbf{y}; \mathcal{P}_{\cK})$ and $M_{r-r_{\cU}}(\mathbf{y_u}; \mathcal{P}_{\cU})$ are localization matrices constructed by polynomials  $\cP_{\cK}$ and $\cP_{\cU}$.

Now, consider the following theorem. 

\begin{theorem} \label{Theo 3}
	The sequence of optimal solutions to the finite SDP in \eqref{Problem4} converges to the moment sequence of measures that are optimal to the infinite LP in \eqref{Problem3}. Hence, $\hbox{lim}_{r \rightarrow \infty} \mathbf{P^*_r} = \mathbf{P^*_{measures}} $.
\end{theorem}

\begin{proof}
	Using Lemma \eqref{lem1} and \eqref{lem2}, the constraints of problem \eqref{Problem4} implies that the sequence of $\mathbf{y}$ and $\mathbf{y_u}$ are the moment sequence of the measures of problem \eqref{Problem3}. For more details, see Lemma 3.2 and Theorem 3.3 in \cite{Ref_Chance}. 
\end{proof}

As in Theorem \ref{Theo 2} and Theorem \ref{Theo 3}, if equivalent problem on measures has delta distribution solution $\mu_u^*$, then problems on measures and moments in \eqref{Problem3} and \eqref{Problem4} are equivalent to the chance constraint problem \eqref{Problem2} and the optimal distribution $\mu_u^*$ is a delta distribution whose mass is concentrated on the single point $u^*$, i.e., its support is the singleton $\{u^*\}$. Such distributions, have moment matrices with rank one.
Hence, we incorporate this observation in the formulation of the relaxed problem \eqref{Problem4} as follows:
\begin{align}
\mathbf{P^*_{trace}}:= & \min_{\mathbf{y},\ \mathbf{y_u}} L_\mathbf{y_u}(\cP_{E}\left( u)\right)+\omega_r\hbox{\textbf{Tr}}(M_r ({\mathbf y}_{\mathbf u})),\label{Problem5}\\
\hbox{s.t.}\quad & \eqref{Problem4_1},\eqref{Problem4_2},\eqref{Problem4_3},\eqref{Problem4_4} \subeqn
\end{align}
where, \textbf{Tr}(.) is the trace function and $\omega_r > 0$. We want to minimize the expected value with a low rank momnet matrix $M_r ({\mathbf y^*}_{\mathbf u})$. For this, we use the trace norm (nuclear norm) which is the convex envelope of the rank function, (\cite{Ref_nuc1},\cite{Ref_nuc2}). Since, $M_r ({\mathbf y^*}_{\mathbf u}) \succcurlyeq 0 $, $\hbox{\textbf{Tr}}(M_r ({\mathbf y^*}_{\mathbf u}))$ is equal to sum of singular values of $M_r ({\mathbf y^*}_{\mathbf u})$.

\textbf{Remark} To be able to apply the provided chance constrained model predictive control to large scale systems, we can implement Fast MPC approach \cite{Ref_Fast} where, one needs to compute the control input $u_k$ offline for all possible states $x_k$. Then, the online controller can be implemented as a lookup table, (see \cite{Ref_Fast} for more details).

\section{Numerical results}\label{Sec:exa}

In this section, two numerical examples are presented that illustrate the
performance of the proposed method. To solve proposed SDP in \eqref{Problem4}, GloptiPoly is employed which is a Matlab-based toolbox aimed at optimizing moments of measures \cite{Ref_Glop}. Using GloptiPoly, we call Mosek [53], which is an interior-point solver add-on for Matlab.

\textbf{Example 1:}
Consider the unstable nonlinear system as
\begin{equation}
\label{eq:control_system}
\begin{array}{r l}
x_1(k+1)=&x_2(k),\\
x_2(k+1)=&x_1(k)x_2(k)+\omega(k)+u(k)
\end{array}
\end{equation}
where, $\chi=[-1,1]^2$ and disturbance {\small $\omega_k \sim U[-0.5, 0.5]$} are uniformly distributed. The desired set is a circle centered at the origin with radius 0.2; hence {\small  $ \chi_D = \lbrace x \in \chi : \mathcal{P}_{\chi_D}(x)= x_1^2 + x_2^2 -0.2^2  \leq 0 \rbrace$}. The finite cost function is defined as {\small $\cP_{cost} = \sum_{i=k}^{k+N_p} \|x(i)\|_{2}^2 + \sum_{i=k}^{k+N_p} \|u(i)\|_{2}^2$}, where $\|.\|_2$ is L-2 norm and {\small $N_p=3$}
To obtain control input, we solve the SDP in \eqref{Problem5} for {\small $\alpha=0.8$}, {\small $\beta=0.0510$}, {\small $\omega_r = 1$}, and relaxation order $r=5$. The obtained control input at each time $k$ for the initial condition $x_0=(1,1)$ is \\
{\small $$u_k=[-0.5634, -0.4647, 0.0007]$$}
where results in the trajectory of \\
{\small  $$x_1(k)=[1,1,0.878,-0.0430]$$}
{\small $$x_2(k)=[1, 0.878,-0.0430,-0.168]$$}
Hence in 3 steps the trajectory of the system under control reaches the desired set. The observed disturbance is {\footnotesize $\omega_k=[0.4421, -0.4570, -0.1315]$}. Also, by applying the obtained control input $u_k$, the cost function at time {\small $k$, $\|x(k)\|_{2}^2 + \|u(k)\|_{2}^2$} is as {\footnotesize $[3.11, 2.56, 0.408]$} and also the trace of the moment matrix is as {\footnotesize $[1.58, 1.37, 1.00]$}. Moreover, the lower bound probability {\small $1-\beta\cP_{D}(x_k)$} and the obtained probability of the event {\small $\left\lbrace \cP_{\chi_D} (x_{k+1}) \geq \alpha \cP_{\chi_D} (x_{k}) \right\rbrace$} is as {\footnotesize $[0.5, 0.558, 0.812]$}.
Note that, we stop the optimization problem and input control by reaching the desired set.  
We can add extra constraint that makes the given desired set, an invariant set; hence the trajectories of the system remains in the set despite all disturbance and uncertainties, (See \cite{Ref_ProbCon} for more details).\\

\textbf{Example 2: }

Consider the uncertain nonlinear system as
\begin{equation}
\label{eq:control_system}
\begin{array}{r l}
x_1(k+1)=& x_2(k),\\
x_2(k+1)=&x_1(k)~x_3(k),\\
x_3(k+1)=&x_1(k)-x_2(k)+x_3(k)+\omega(k)+u(k)
\end{array}
\end{equation}
where, {\small $\chi=[-1,1]^3$} and disturbances {\small $\omega(k) \sim U[-0.5, 0.5]$} are uniformly distributed. Also, 
The desired set is a circle centered at the origin with radius 0.2.The finite cost function is defined as {\footnotesize $\cP_{cost} = \sum_{i=k}^{k+N_p} \|x(i)\|_{2}^2 + \sum_{i=k}^{k+N_p} \|u(i)\|_{2}^2$}, where $\|.\|_2$ is L-2 norm and {\small $N_p=3$}.
To obtain control input, we solve the SDP in \eqref{Problem5} for {\small $\alpha=0.9$}, {\small $\beta= 0.2027$}, {\small $\omega_r = 1$}, and relaxation order $r=5$. 
The obtained control input at each time $k$ for the initial condition {\small $x_0=(1,1,1)$} is 
{\footnotesize $$u_k=[ -0.227,-0.219,-0.325,-0.196,-0.215,-0.605,0.550]$$}where results in the trajectory of 
{\footnotesize $$x_1(k)=[1,1,1,0.752,0.892,0.417,-0.101,0.0487]$$}
{\footnotesize $$x_2(k)=[1,1,0.752,0.892,0.417,-0.101,0.0487,0.041]$$}
{\footnotesize $$x_3(k)=[1,0.752,0.892,0.554,-0.113,0.116,-0.410,0.171]$$}
Hence in 7 steps the trajectory of the system under control reaches the desired set. The observed disturbance is 
{\footnotesize   $$\omega_k=[  -0.020,0.359,   -0.260,   -0.332,   -0.028,   -0.440,    0.182]$$}
Also, by applying the obtained control input $u_k$, the cost function at time $k$, $\|x(k)\|_{2}^2 + \|u(k)\|_{2}^2$ is as {\footnotesize $[7.61,5.33,5.86,2.95,1.6,1.61,1.45]$} and also the trace of the moment matrix is as {\footnotesize $[   1.26,    1.22,    1.34 ,   1.16,    1.12 ,   1.65,    1.68]$}. Moreover, the lower bound probability {\small $1-\beta\cP_{D}(x_k)$} and the obtained probability of the event {\small $\left\lbrace \cP_{\chi_D} (x_{k+1}) \geq \alpha \cP_{\chi_D} (x_{k}) \right\rbrace$} is as {\footnotesize $[0.5,0.573,0.607,0.724,0.840,0.973,0.976]$}.

\section{CONCLUSION}

In this paper, chance constrained model predictive control problems are addressed, where one aims at
finding optimal control input to minimize expected value of given cost function with respect to probabilistic constraints. These problems are, in general, nonconvex and computationally hard. Using theory of measures and moments, a sequence of semidefinite relaxations is provided whose sequence of optimal values is shown to converge to the optimal value of the original problem. Numerical examples are provided that show that one can obtains reasonable approximations to the optimal solution.


\appendices


\section{Proof of Theorem \ref{Theo1}}\label{Appen_Theo1}
Given the system in \eqref{sys1}, the desired set $\chi_D$, and the initial state $x_0 \in \chi$, the condition {\small $\hbox{Prob}_{\mu_{\omega_k}}\left\lbrace \cP_{\chi_D} (x_{k+1}) \leq \alpha \cP_{\chi_D} (x_{k}) \right\rbrace \geq 1- \beta \cP_{\chi_D} (x_{k}) $ }is satisfied at each sampling time $k$, where $0< \alpha, \beta <1$. For a given $\hat{k}$, we define the events $\chi_1$ and $\chi_2$ as follow:
\begin{equation}
\chi_1=\{(x_0,...,x_{\hat{k}}): \cP_{\chi_D} (x_{\hat{k}}) \leq \epsilon \}
\end{equation}
{\small  \begin{equation}
\chi_2=\{(x_0,...,x_{\hat{k}}): \cP_{\chi_D} (x_{i+1}) \leq \alpha\cP_{\chi_D} (x_{i}),\ i=0,...,\hat{k}-1 \}
\end{equation}}where, $\alpha\cP_{\chi_D} (x_{\hat{k}-1}) \leq \epsilon$ and; hence, $\alpha^{\hat{k}}\cP_{\chi_D} (x_{0}) \leq \epsilon$. 
This implies that given $x_0$, $\epsilon$, and $\alpha$, the time $\hat{k}$ for which $\cP_{\chi_D} (x_{\hat{k}}) \leq \epsilon$ has lower bound of
\begin{equation}
\hat{k}\geq \frac{\hbox{ln}(\epsilon)-\hbox{ln}(\cP_{\chi_D} (x_{0}) )}{\hbox{ln}(\alpha)} \label{Klower}
\end{equation}
Also, $\chi_2 \subset \chi_1$ and thus $\hbox{Prob}(\chi_2) \leq \hbox{Prob}(\chi_1)$.
Since, the distribution of the uncertain parameters and disturbance at each time $k$ are independent, the stochastic model \eqref{sys1} has Markov property; hence, the probability of the event $\chi_2$ is
\begin{equation}\label{Theo1_1}
\hbox{Prob}\left\lbrace \chi_2 \right\rbrace = \prod_{i=0}^{\hat{k}-1} \hbox{Prob}\left\lbrace  \cP_{\chi_D} (x_{i+1}) \leq \alpha\cP_{\chi_D} (x_{i})  | x_i \right\rbrace
\end{equation}
The probability in \eqref{Theo1_1} has lower bound as
\begin{equation}\label{Theo1_2}
\hbox{Prob}\left\lbrace \chi_2 \right\rbrace  \geq \prod_{i=0}^{\hat{k}-1}(1- \beta \cP_{\chi_D} (x_{i})) \geq \prod_{i=0}^{\hat{k}-1}(1- \beta \alpha^i)
\end{equation}
where, $\cP_{\chi_D}(x_i) \leq \alpha^i \cP_{\chi_D}(x_0) $.
Hence, the lower bound of probability read as
\begin{equation}
\hat{P}(\epsilon, \alpha, \beta)=\prod_{i=0}^{\hat{k}-1}(1- \beta \alpha^i)
\end{equation}
This is a convergent product and converges to nonzero constant as $\hat{k} \rightarrow \infty$. As $\epsilon \rightarrow 0$, by \eqref{Klower} $\hat{k} \rightarrow \infty$; hence, $\hat{P}$ is non-zero and bounded.


\section{Proof of Theorem \ref{Theo 2}}\label{Appen_Theo2}

Consider the following problem over the measures $\mu_u$

{\scriptsize \begin{align} \label{Theo2_1}
	& \mathbf {P_{\mu_u}} :=\ \min_{\mu_u\in\cM_+(\mathcal{U})} \int_{\mathcal{U}}  \cP_{E}(u) d\mu_u  \\
	& \hbox{s.t.}\quad \nonumber \\
	& \int_{\mathcal{U}}  \hbox{Prob}\left\lbrace \cP_{\chi_D} (x_{k+1}) \geq \alpha \cP_{\chi_D} (x_{k}) \right\rbrace   d\mu_u \geq \int_{\mathcal{U}} (1- \beta \cP_{\chi_D} (x_{k})) d\mu_u  \label{Theo2_11} \subeqn \\
	& \int d\mu_u =1,\ \ \{\omega_i \sim \mu_{\omega_i} \}_{i=k}^{k+N_p-1}\subeqn
	\end{align}}
We first want to show that $\mathbf{P_{MPC}^*} = \mathbf{P_{\mu_u}}$. Let $\mu_u$ be a feasible solution to \eqref{Theo2_1}, i.e., {\footnotesize $\int_{\mathcal{U}}  \hbox{Prob}\left\lbrace \cP_{\chi_D} (x_{k+1}) \geq \alpha \cP_{\chi_D} (x_{k}) \right\rbrace   d\mu_u \geq 1- \beta \cP_{\chi_D} (x_{k})) $}. Then for any $u$ in support of measure $\mu_u$  {\footnotesize $\hbox{Prob}\left\lbrace \cP_{\chi_D} (x_{k+1}) \geq \alpha \cP_{\chi_D} (x_{k}) \right\rbrace  \geq 1- \beta \cP_{\chi_D} (x_{k})) $}, i.e, the feasible set of problem \eqref{Problem2}. Also, Since, {\footnotesize $\cP_{E}(u)   \le \mathbf{P_{MPC}^*}$} for all $u \in \mathcal{U}$, we have{\footnotesize  $\int_{\mathcal{U}}  \cP_{E}(u) d\mu_u  \le \mathbf{P_{MPC}^*}$}. Thus, {\footnotesize $\mathbf{P_{\mu_u}} \le \mathbf{P_{MPC}^*}$}. Conversely, let $u\in\mathcal{U}$ be a feasible solution to the problem in \eqref{Problem2}. Let $\delta_{u}$ denotes the Dirac measure at $u$. Then the $\delta_{u}$ belongs to the feasible set of problem \eqref{Theo2_1}.
The objective value of $u$ in \eqref{Problem2} is equal to $\cP_{E}(u)$. 
Moreover, $\mu_u = \delta_{u}$ is a feasible solution to the problem in \eqref{Theo2_1} with objective value equal to $\cP_{E}(u)$. This implies that $\mathbf{P_{MPC}^*} \le \mathbf{P_{\mu_u}}$. Hence, $\mathbf{P_{MPC}^*} = \mathbf{P_{\mu_u}}$, and \eqref{Theo2_1} can be rewritten as
{\small \begin{align} \label{Theo2_2}
	& \mathbf {P_{\mu_u}} :=\ \min_{\mu_u\in\cM_+(\mathcal{U})} \int_{\mathcal{U}}  \cP_{E}(u) d\mu_u  \\
	& \hbox{s.t.}\quad \nonumber \\
	& \int_{\mathcal{U}} \int_{\cK} d\mu_u d\mu \geq 1- \beta \cP_{\chi_D} (x_{k})  \label{Theo2_22} \subeqn \\
	& \int d\mu_u =1,\ \ \{\omega_i \sim \mu_{\omega_i} \}_{i=k}^{k+N_p-1}\subeqn
	\end{align}}where, set $\cK$ is defined in \eqref{K1}.
Using the Lemma \ref{preliminary result_volume}, we obtain 
\begin{align}
\mathbf{P_{measure}^*}:=&\ \min_{\mu ,\mu_u} \int  \cP_{E}(u)  d\mu_u,  \\
& \hbox{s.t.}\  \int d\mu \geq (1- \beta \cP_{\chi_D} (x_{k})) \label{Pro1} \subeqn\\
& \mu \preccurlyeq \mu_u \times \Pi_{i=k}^{k+N_p-1}\mu_{\omega_i}, \subeqn\\
&\int d\mu_u = 1, \label{Pro2} \subeqn\\
& \mu\in\cM_+(\mathcal K),\ \mu_u\in\cM_+(\cU).  \subeqn
\end{align}
Note that, if there exist delta solution $\mu_{u}^*$ for the problem \eqref{Problem3} whose mass is concentrated on a single point $u^*$, the $\int d\mu$ in constraint \eqref{Pro1} implies the probability of event $\left\lbrace \cP_{\chi_D} (x_{k+1}) \geq \alpha \cP_{\chi_D} (x_{k}) \right\rbrace$ for a control input $u^*$. 

\end{document}

%% file: defs.tex
\def\cB{\mathcal{B}}

\def\cK{\mathcal{K}}

\def\cM{\mathcal{M}}

\def\cP{\mathcal{P}}

\def\cU{\mathcal{U}}

\def\smskip{\smallskip}

\def\texitem#1{\par\smskip\noindent\hangindent 25pt
               \hbox to 25pt {\hss #1 ~}\ignorespaces}

\def\norm#1{\|#1\|}

\newcommand{\BEAS}{\begin{eqnarray*}}
\newcommand{\EEAS}{\end{eqnarray*}}
\newcommand{\BEA}{\begin{eqnarray}}
\newcommand{\EEA}{\end{eqnarray}}
\newcommand{\BEQ}{\begin{eqnarray}}
\newcommand{\EEQ}{\end{eqnarray}}
\newcommand{\BIT}{\begin{itemize}}
\newcommand{\EIT}{\end{itemize}}
\newcommand{\BNUM}{\begin{enumerate}}
\newcommand{\ENUM}{\end{enumerate}}

\newcommand{\BA}{\begin{array}}
\newcommand{\EA}{\end{array}}

\newcommand{\reals}{\mathbb{R}}
\newcommand{\integers}{\mathbb{Z}}

\newif\ifpagenumbering
\pagenumberingtrue

\pagenumberingfalse

\newsavebox{\theorembox}
\newsavebox{\lemmabox}
\newsavebox{\remarkbox}
\newsavebox{\assbox}
\savebox{\theorembox}{\noindent\bf Theorem}
\savebox{\lemmabox}{\noindent\bf Lemma}
\savebox{\remarkbox}{\noindent\bf Remark}
\savebox{\assbox}{\noindent\bf Assumption}

